\documentclass[12pt,reqno]{amsart}

\usepackage{textcomp}

\usepackage{amssymb}

\usepackage{mathrsfs}

\DeclareMathAlphabet{\mathpzc}{OT1}{pzc}{m}{it}

\usepackage[all]{xy}

\entrymodifiers={+!!<0pt,\fontdimen22\textfont2>}

\def\cB{\mathscr{B}}

\def\cF{\mathscr{F}}
\def\cG{\mathscr{G}}

\def\BQ{\mathbb{Q}}

\def\fm{\mathfrak{m}}

\def\sD{\mathsf{D}}

\def\adots{\mathinner{\mkern1mu\raise1.0pt\vbox{\kern7.0pt\hbox{.}}\mkern2mu\raise4.0pt\hbox{.}\mkern2mu\raise7.0pt\hbox{.}\mkern1mu}}

\def\Coker{\operatorname{Coker}}

\def\depth{\operatorname{depth}}

\def\Ext{\operatorname{Ext}}

\def\gldim{\operatorname{gldim}}

\def\H{\operatorname{H}}

\def\Hom{\operatorname{Hom}}
\def\id{\operatorname{id}}
\def\Image{\operatorname{Im}}

\def\Ker{\operatorname{Ker}}

\def\mod{\mathsf{mod}}

\def\RHom{\operatorname{RHom}}

\numberwithin{equation}{section}

\newtheorem{Lemma}{Lemma}[section]
\newtheorem{Theorem}[Lemma]{Theorem}
\newtheorem{Proposition}[Lemma]{Proposition}

\theoremstyle{definition}

\newtheorem{Setup}[Lemma]{Setup}

\newtheorem{Construction}[Lemma]{Construction}
\newtheorem{Remark}[Lemma]{Remark}

\newtheorem{Example}[Lemma]{Example}

\begin{document}

\setlength{\parindent}{0pt}
\setlength{\parskip}{7pt}

\title[Rings without Gorenstein Govorov-Lazard]
{Rings without a Gorenstein analogue of the Govorov-Lazard Theorem}

\author{Henrik Holm}
\address{Department of Basic Sciences and Environment, Faculty of Life
  Sciences, University of Co\-pen\-ha\-gen, Thorvaldsensvej 40, 6th
  floor, 1871 Frederiksberg C, Denmark}
\email{hholm@life.ku.dk}
\urladdr{http://www.dina.kvl.dk/\~{ }hholm/}

\author{Peter J\o rgensen}
\address{School of Mathematics and Statistics,
Newcastle University, Newcastle upon Tyne NE1 7RU,
United Kingdom}
\email{peter.jorgensen@ncl.ac.uk}
\urladdr{http://www.staff.ncl.ac.uk/peter.jorgensen}





\begin{abstract} 

It was proved by Beligiannis and Krause that over certain Artin
algebras, there are Gorenstein flat modules which are not direct
limits of finitely generated Gorenstein projective mo\-du\-les.  That
is, these algebras have no Gorenstein analogue of the
Go\-vo\-rov-La\-zard Theorem.

We show that, in fact, there is a large class of rings without such an
analogue.  Namely, let $R$ be a commutative local noetherian ring.
Then the analogue fails for $R$ if it has a dualizing complex, is
henselian, not Gorenstein, and has a finitely generated Gorenstein
projective module which is not free.

The proof is based on a theory of Gorenstein projective
(pre)en\-ve\-lo\-pes.  We show, among other things, that the finitely
generated Gorenstein projective modules form an enveloping class in
$\mod\, R$ if and only if $R$ is Gorenstein or has the property that
each finitely generated Gorenstein projective module is free.

This is analogous to a recent result on covers by Christensen,
Piepmeyer, Striuli, and Takahashi, and their methods are an important
input to our work.

\end{abstract}

\maketitle

\setcounter{section}{-1}
\section{Introduction}
\label{sec:introduction}

{\em Gorenstein homological algebra} was founded by Auslander and
Bridger in \cite{AB}.  Some of its main concepts are the so-called
Gorenstein projective and Gorenstein flat modules, see
\cite{EJpaper} and \cite{EJT}.  These modules inhabit a theory
parallel to classical homological algebra.  For instance, just as
projective modules can be used to define projective dimension, so
Gorenstein projective modules can be used to define Gorenstein
projective dimension.  A commutative local noetherian ring is
Gorenstein if and only if all its modules have finite Gorenstein
projective dimension.  A good introduction is given in \cite{LWC}; in
particular, the definitions of Gorenstein projective and Gorenstein
flat modules can be found in \cite[(4.2.1) and (5.1.1)]{LWC}.

{\em The Govorov-Lazard Theorem} says that the closure under direct
limits of the class of finitely generated projective modules is equal
to the class of flat modules; see \cite{Govorov} and \cite[thm.\
1.2]{Lazard}.  It is natural to ask if this has a Gorenstein analogue.
Namely, if $\cG$ denotes the class of finitely generated Gorenstein
projective modules, is $\varinjlim \cG$ equal to the class of
Gorenstein flat modules?  In some cases the answer is yes, for
instance over a ring which is Gorenstein in a suitable sense; this was
established by Enochs and Jenda in \cite[thm.\ 10.3.8]{EJ}.  However,
Beligiannis and Krause proved in \cite[4.2 and 4.3]{BK} that for
certain Artin algebras, the answer is no.

We show for a considerably larger class of rings that there is no
Gorenstein analogue of the Govorov-Lazard Theorem.  Namely, let $R$ be
a commutative local noetherian ring and let $\cF$ be the class of
finitely generated free modules.  The following is our Theorem
\ref{thm:limG}.

\noindent
{\bf Theorem A.}
{\it  If $R$ has a dualizing complex, is henselian, not Gorenstein,
and has $\cG \neq \cF$, then $\varinjlim \cG$ is strictly contained in 
the class of Gorenstein flat modules.
}

The proof is based on a theory of $\cG$-preenvelopes, the development
of which takes up most of the paper.  The background is that the
existence of $\cG$-precovers has been considered at length.  That is,
if $M$ is a finitely generated module, does there exist a homomorphism
$G \stackrel{\gamma}{\rightarrow} M$ with $G$ in $\cG$ such that any
other homomorphism $G^{\prime} \rightarrow M$ with $G^{\prime}$ in
$\cG$ factors through $\gamma$?  A breakthrough was achieved recently
in \cite{CPST} by Christensen, Piepmeyer, Striuli, and Takahashi who
proved, among other things, that if $R$ is henselian, then
$\cG$-precovers exist for all finitely generated modules in precisely
two cases: If $R$ is Gorenstein, or if $\cG = \cF$.

We will consider the dual question: Existence of
$\cG$-pre\-en\-ve\-lo\-pes.  That is, if $M$ is a finitely generated
module, does there exist a homomorphism $M \stackrel{\mu}{\rightarrow}
G$ with $G$ in $\cG$ such that any other homomorphism $M \rightarrow
G^{\prime}$ with $G^{\prime}$ in $\cG$ factors through $\mu$?  We give
criteria for the existence of various types of $\cG$-preenvelopes in
Theorem \ref{thm:env}.  One aspect is the following precise analogue
of the precovering case.

\noindent
{\bf Theorem B.}
{\it
If $R$ is henselian then all finitely generated $R$-modules have
$\cG$-preenvelopes if and only if $R$ is Gorenstein or $\cG = \cF$.
}

Note that the methods and results of \cite{CPST} are an important
input to our proof.

The paper is organized as follows: Section \ref{sec:duals} prepares
the ground by examining the connections between $\cG$-precovers and
$\cG$-preenvelopes which are induced by the algebraic duality functor
$(-)^* = \Hom_R(-,R)$.  Section \ref{sec:preenvelopes} proves Theorems
A and B, among other things.  Section \ref{sec:special} shows a method
for constructing a Gorenstein flat module outside $\varinjlim \cG$.

\section{Algebraic duals of precovers and preenvelopes}
\label{sec:duals}

This section proves Theorems \ref{thm:cover_gives_envelope} and
\ref{thm:envelope_gives_cover} by which algebraic duals of various
types of $\cG$-precovers give the corresponding types of
$\cG$-preenvelopes, and vice versa.

\begin{Setup}
\label{set:1}
Throughout the paper, $R$ is a commutative noetherian ring.
\end{Setup}

By $\mod\, R$ is denoted the category of finitely generated
$R$-modules.  Recall that $\cF$ is the class of finitely generated
free $R$-modules and $\cG$ is the class of finitely generated
Gorenstein projective $R$-modules.

\begin{Remark}
\label{rmk:G}
The following properties of $\cG$ will be used below.
\begin{enumerate}

  \item  $\Ext_R^{\geqslant 1}(\cG,R) = 0$.

\smallskip

  \item  $R$ is in $\cG$.

\smallskip

  \item  The class $\cG$ is closed under the algebraic duality functor
         $(-)^* = \Hom_R(-,R)$. 

\smallskip

  \item  The biduality homomorphism $G
         \stackrel{\delta_G}{\longrightarrow} G^{**}$, as defined in
         \cite[(1.1.1)]{LWC}, is an isomorphism for each $G$ in $\cG$.

\smallskip

  \item  Each module in $\cG$ is isomorphic to a module $G^*$ where
         $G$ is in $\cG$. 

\end{enumerate}
Here (i) and (iv) are part of the definition of $\cG$, see
\cite[def.\ (1.1.2)]{LWC}.  (ii) is by \cite[rmk.\ (1.1.3)]{LWC} and
(iii) is by \cite[obs.\ (1.1.7)]{LWC}.  (v) is immediate from (iii) and
(iv).
\end{Remark}

\begin{Lemma}
\label{lem:Ext}
If $C$ is an $R$-module satisfying $\Ext_R^1(C,R)=0$, then
$\Ext_R^1(G,C^*) \cong \Ext_R^1(C,G^*)$ for each $G$ in $\cG$.
\end{Lemma}

\begin{proof}
We have
\begin{equation}
\label{equ:a}
  \H^{<0}\!\RHom(C,R) = 0,
\end{equation}
so $\RHom(C,R)$ can be represented in the derived category $\sD(R)$ by
a complex concentrated in non-negative cohomological degrees.  Hence
there is a canonical morphism in $\sD(R)$ from the zeroth cohomology
$\H^0\!\RHom(C,R) \cong C^*$ to $\RHom(C,R)$.  Complete it to a
distinguished triangle,
\begin{equation}
\label{equ:c}
  C^* \stackrel{\chi}{\rightarrow} \RHom(C,R) \rightarrow M \rightarrow, 
\end{equation}
and consider the long exact cohomology sequence which consists of
pieces 
\[
   \H^i(C^*)
   \stackrel{\H^i\!\chi}{\longrightarrow} \H^i\!\RHom(C,R)
   \longrightarrow \H^i\!M.
\]
Since $C^*$ is a module, $\H^i(C^*) = 0$ for $i \neq 0$.  Combined with
equation \eqref{equ:a}, the long exact sequence hence implies
$\H^{\leqslant -2}\!M = 0$.

Moreover, $\H^0\!\chi$ is an isomorphism by the construction of
$\chi$, and by assumption, $\H^1\!\RHom(C,R) = \Ext^1(C,R) = 0$.  So
in fact, the long exact sequence also implies $\H^{-1}\!\!M =
\H^0\!M = \H^1\!\!M = 0$.

Consequently, the complex $M$ admits an injective resolution of the
form $I = \cdots \rightarrow 0 \rightarrow I^2 \rightarrow I^3
\rightarrow \cdots$, and in particular,
\begin{equation}
\label{equ:b}
  \H^{\leqslant 1}\!\RHom(G,M) \cong \H^{\leqslant 1}\!\Hom(G,I) = 0
\end{equation}
for each $R$-module $G$.

Now let $G$ be in $\cG$. It follows from Remark \ref{rmk:G}(i) that
there is an isomorphism $\RHom(G,R) \cong G^*$ in $\sD(R)$, and hence
by ``swap'', \cite[(A.4.22)]{LWC}, we get
\[
  \RHom(G,\RHom(C,R)) \cong \RHom(C,\RHom(G,R)) \cong \RHom(C,G^*).
\]
Thus, by applying $\RHom(G,-)$ to the distinguished triangle
\eqref{equ:c} we obtain
\[
  \RHom(G,C^*) \rightarrow \RHom(C,G^*) \rightarrow \RHom(G,M) \rightarrow.
\]
Combining the long exact cohomology sequence of this with equation
\eqref{equ:b} proves the lemma.
\end{proof}

\begin{Lemma}
\label{lem:Ext-vanishing}
Let $C$ be an $R$-module.
\begin{enumerate}

  \item  If $\Ext_R^1(C,\cG)=0$ then $\Ext_R^1(\cG,C^*)=0$.

\smallskip

  \item  If $\Ext_R^1(C,R)=0$ and $\,\Ext_R^1(\cG,C^*)=0$, then
  $\Ext_R^1(C,\cG)=0$.

\end{enumerate}
\end{Lemma}

\begin{proof}
Combine Lemma \ref{lem:Ext} with Remark \ref{rmk:G}, parts (ii) and
(iii), respectively, part (v).
\end{proof}

Let $G \stackrel{\gamma}{\rightarrow} N$ be a $\cG$-precover.  For the
following theorems, recall that $\gamma$ is called a special
$\cG$-precover if $\Ext_R^1(\cG,\Ker\,\gamma) = 0$, and that $\gamma$
is called a cover if each endomorphism $G
\stackrel{\varphi}{\rightarrow} G$ with $\gamma\varphi = \gamma$ is an
automorphism.  Special $\cG$-preenvelopes and $\cG$-envelopes are
defined dually.

\begin{Theorem}
\label{thm:cover_gives_envelope}
Let $M$ be in $\mod\, R$, let $G$ be in $\cG$, and let $G
\stackrel{\gamma}{\rightarrow} M^*$ be a homomorphism.  Consider the
composition
\[
  M
  \stackrel{\delta_M}{\longrightarrow} M^{**}
  \stackrel{\gamma^*}{\longrightarrow} G^*
\]
where $\delta$ denotes the biduality homomorphism again.  Then
\begin{enumerate}

  \item  If $\gamma$ is a $\cG$-precover then $\gamma^*\delta_M$ is
         a $\cG$-preenvelope. 

\smallskip

  \item  If $\gamma$ is a special $\cG$-precover then
         $\gamma^*\delta_M$ is a special $\cG$-pre\-en\-ve\-lo\-pe.

\smallskip

  \item  If $\gamma$ is a $\cG$-cover then $\gamma^*\delta_M$
         is a $\cG$-envelope.

\end{enumerate}
\end{Theorem}

\begin{proof}
There is a commutative diagram
\[
  \hspace{14ex}
  \begin{gathered}
    \xymatrix @R7ex @C7ex {
      G \ar[d]^-{\cong}_-{\delta_G} \ar[r]^-{\gamma} 
        & M^* \ar@<0.75ex>@{^(->}[d]^-{\delta_{M^*}} \\ 
      G^{**} \ar[r]_-{\gamma^{**}}
        & M^{***} \ar@<0.75ex>@{->>}[u]^-{(\delta_M)^*} 
             }
  \end{gathered}
  \mspace{-150mu}
  \begin{split}
    (1)\;\; & \delta_{M^*}\gamma = \gamma^{**}\delta_G, \\
    (2)\;\; & \gamma = (\delta_M)^*\gamma^{**}\delta_G.
  \end{split}
\]
Here (1) just says that the biduality homomorphism is natural.  By the
proof of \cite[prop.\ (1.1.9)]{LWC} we have
$(\delta_M)^*\delta_{M^*}=1_{M^*}$, so $\delta_{M^*}$ is (split)
injective, $(\delta_M)^*$ (split) surjective.  Now (2) follows from
$\delta_{M^*}(\delta_M)^*\gamma^{**}\delta_G =
\delta_{M^*}(\delta_M)^*\delta_{M^*}\gamma = \delta_{M^*}\gamma$
since $\delta_{M^*}$ is injective.

(i).  Suppose that $\gamma$ is a $\cG$-precover and let
$\widetilde{G}$ be in $\cG$. Remark \ref{rmk:G}(iv) and ``swap'' in
the form \cite[II.\ Exer.\ 4]{CE} give the following natural
e\-qui\-va\-len\-ces of functors,
\[
  \Hom(-,\widetilde{G}) 
  \simeq \Hom(-,\widetilde{G}^{**})
  \simeq \Hom(\widetilde{G}^*,(-)^*).
\]
This gives the (top) two squares of the commutative diagram below,
where we have abbreviated $\Hom(-,-)$ to $(-,-)$.  The (bottom)
commutative triangle comes from applying $\Hom(\widetilde{G}^*,-)$ to
part (2) from the beginning of the proof.
\[
  \xymatrix @C=10ex {
    (G^*,\widetilde{G}) \ar[d]^-{\cong}
      \ar[r]^-{(\gamma^*,\widetilde{G})}
    & (M^{**},\widetilde{G}) \ar[d]^-{\cong}
      \ar[r]^-{(\delta_M,\widetilde{G})}
    & (M,\widetilde{G}) \ar[d]^-{\cong} \\
    (\widetilde{G}^*,G^{**}) \ar[r]^-{(\widetilde{G}^*,\gamma^{**})}
    & (\widetilde{G}^*,M^{***}) \ar[r]^-{(\widetilde{G}^*,(\delta_M)^*)}
    & (\widetilde{G}^*,M^*) \\ {}
    & (\widetilde{G}^*,G) \ar[ul]_-{\cong}^-{(\widetilde{G}^*,\delta_G)}
      \ar@{->>}[ur]_-{(\widetilde{G}^*,\gamma)} & {} 
                    }
\]
Since $\widetilde{G}^*$ is in $\cG$ by Remark \ref{rmk:G}(iii), the
map $\Hom(\widetilde{G}^*,\gamma)$ is surjective, and the diagram
implies that so is $\Hom(\delta_M,\widetilde{G}) \circ
\Hom(\gamma^*,\widetilde{G}) = \Hom(\gamma^*\delta_M,\widetilde{G})$.
Hence $\gamma^*\delta_M$ is a $\cG$-preenvelope.

(ii).  Suppose that $\gamma$ is a special $\cG$-precover; in
particular we have $\Ext^1(\cG,\Ker \gamma) = 0$.  Part (i) says that
$\gamma^*\delta_M$ is a $\cG$-pre\-en\-ve\-lo\-pe, and it remains to
show $\Ext^1(C,\cG)=0$ where $C = \Coker(\gamma^*\delta_M)$. To prove
this we use Lemma \ref{lem:Ext-vanishing}(ii).  Thus we need to show
that $\Ext^1(\cG,C^*)=0$ and $\Ext^1(C,R)=0$.

Applying $(-)^*$ to the exact sequence $\xymatrix{M
\ar[r]^-{\gamma^*\delta_M} & G^* \ar[r]^{\pi} & C \ar[r] & 0}$ gives the
second exact row in
\[
    \xymatrix @C8ex {
      {} & {} & G \ar[d]^-{\cong}_-{\delta_G} \ar[r]^-{\gamma} & M^* \ar@{=}[d] \\
      0 \ar[r] & C^* \ar[r]_{\pi^*} & G^{**} \ar[r]_-{(\delta_M)^*\gamma^{**}} & M^*
                     }
\]
where the square is commutative by part (2) at the beginning of the
proof.  It follows that $C^* \cong \Ker\gamma$, and hence
$\Ext^1(\cG,C^*)=0$.

To prove $\Ext^1(C,R)=0$, we will argue that each short exact sequence
$0 \rightarrow R \rightarrow E \rightarrow C \rightarrow 0$
splits. Consider the diagram with exact rows,
\[
  \xymatrix @R8ex @C8ex {
    {}
    & M \ar@{-->}[d]_-{\mu} \ar[r]^-{\gamma^*\delta_M}
    & G^* \ar@{-->}[dl]_-{\varphi} \ar@{-->}[d]_-{\nu} \ar[r]^-{\pi}
    & C \ar@{-->}[dl]_-{\chi} \ar@{=}[d] \ar[r]
    & 0 \\
    0 \ar[r] 
    & R \ar[r]_-{\rho} 
    & E \ar[r]_-{\varepsilon} 
    & C \ar[r]
    & 0 \lefteqn{.}
            }
\]

By Remark \ref{rmk:G}, (i) and (iii), we have $\Ext^1(G^*,R)=0$, so
the functor $\Hom(G^*,-)$ preserves the exactness of the bottom row.
In particular, there exists $G^* \stackrel{\nu}{\rightarrow} E$ with
$\varepsilon \nu = \pi$. By the universal property of the kernel of
$\varepsilon$, there exists a (unique) $M \stackrel{\mu}{\rightarrow}
R$ with $\rho \mu = \nu \gamma^* \delta_M$.

Since $\gamma^*\delta_M$ is a $\cG$-preenvelope and since $R$ is in
$\cG$ by Remark \ref{rmk:G}(ii), there exists $G^*
\stackrel{\varphi}{\rightarrow} R$ satisfying $\varphi\gamma^*\delta_M
= \mu$. It follows that
\[
  (\nu - \rho\varphi)\gamma^*\delta_M
  = \nu\gamma^*\delta_M - \rho\varphi\gamma^*\delta_M
  = \nu\gamma^*\delta_M - \rho\mu
  = 0,
\]
so by the universal property of the cokernel of $\gamma^*\delta_M$,
there exists a (unique) $C \stackrel{\chi}{\rightarrow} E$ with
$\chi \pi = \nu - \rho \varphi$. Consequently,
\[
  \varepsilon\chi\pi
  = \varepsilon(\nu - \rho\varphi)
  = \varepsilon\nu - \varepsilon\rho\varphi
  = \pi - 0
  = \id_C\pi,
\]
and since $\pi$ is surjective we get $\varepsilon\chi = \id_C$.  This
proves that $\varepsilon$ is a split epimorphism as desired.

(iii).  Suppose that $\gamma$ is a $\cG$-cover.  Part (i) says that
$\gamma^*\delta_M$ is a $\cG$-preenvelope, and it remains to show that
each endomorphism $G^* \stackrel{\varphi}{\rightarrow} G^*$ with
\begin{equation}
\label{equ:d}
  \varphi\gamma^*\delta_M = \gamma^*\delta_M
\end{equation}
is an automorphism.  However, such an endomorphism has
\[
  \gamma\delta_G^{-1}\varphi^*
  = (\delta_M)^*\gamma^{**}\varphi^*
  = (\delta_M)^*\gamma^{**}
  = \gamma\delta_G^{-1}
\]
where the first and third $=$ are by part (2) at the beginning of the
proof while the second $=$ is $(-)^*$ of equation \eqref{equ:d}.
Hence $\gamma(\delta_G^{-1}\varphi^*\delta_G) = \gamma$, and since
$\gamma$ is a $\cG$-cover and $\delta_G^{-1}\varphi^*\delta_G$ is an
endomorphism of $G$, it follows that $\delta_G^{-1}\varphi^*\delta_G$
is an automorphism.

Therefore $\varphi^*$, and hence also $\varphi^{**}$, is an
automorphism.  Applying Remark \ref{rmk:G}, (iii) and (iv), and
naturality of the biduality homomorphism gives $\varphi =
\delta_{G^*}^{-1}\varphi^{**}\delta_{G^*}$ whence $\varphi$ is an
automorphism as desired.
\end{proof}

\begin{Theorem}
\label{thm:envelope_gives_cover}
Let $M$ be in $\mod\, R$, let $G$ be in $\cG$, and let $M
\stackrel{\mu}{\rightarrow} G$ be a homomorphism.  Consider the
algebraic dual $G^* \stackrel{\mu^*}{\rightarrow} M^*$.  Then
\begin{enumerate}

  \item  If $\mu$ is a $\cG$-preenvelope then $\mu^*$ is
         a $\cG$-precover.

\smallskip

  \item  If $\mu$ is a special $\cG$-preenvelope then $\mu^*$
         is a special $\cG$-precover. 

\smallskip

  \item  If $\mu$ is a $\cG$-envelope then $\mu^*$ is a
         $\cG$-cover.

\end{enumerate} 
\end{Theorem}

\begin{proof}
(i).  We have $\Hom(G,\mu^*) \cong \Hom(\mu,G^*)$ by ``swap'',
\cite[II.\ Exer.\ 4]{CE}, and combined with Remark \ref{rmk:G}(iii)
this implies the claim.

(ii).  Suppose that $\mu$ is a special $\cG$-pre\-en\-ve\-lo\-pe; in
particular we have $\Ext^1(\Coker \mu,\cG) = 0$.  Part (i) says that
$\mu^*$ is a $\cG$-pre\-cover, and it remains to show
$\Ext^1(\cG,\Ker(\mu^*)) = 0$.  But this follows from Lemma
\ref{lem:Ext-vanishing}(i) because $\Ker(\mu^*)
\cong (\Coker\mu)^*$.
  
(iii).  Suppose that $\mu$ is a $\cG$-envelope.  Part (i) says that
$\mu^*$ is a $\cG$-pre\-cover, and it remains to show that each $G^*
\stackrel{\varphi}{\rightarrow} G^*$ with $\mu^*\varphi = \mu^*$ is an
automorphism.

The biduality homomorphism is natural so $\delta_G\mu =
\mu^{**}\delta_M$, and since $\delta_G$ is an isomorphism by Remark
\ref{rmk:G}(iv), it follows that $\mu =
\delta_G^{-1}\mu^{**}\delta_M$.  Applying $(-)^*$ to $\mu^*\varphi =
\mu^*$ gives $\varphi^*\mu^{**} = \mu^{**}$.  Combining these gives 
\[
  (\delta_G^{-1}\varphi^*\delta_G)\mu =
  (\delta_G^{-1}\varphi^*\delta_G)(\delta_G^{-1}\mu^{**}\delta_M) =
  \delta_G^{-1}\varphi^*\mu^{**}\delta_M = \delta_G^{-1}\mu^{**}\delta_M
  = \mu.
\]
Since $\mu$ is a $\cG$-envelope and $\delta_G^{-1}\varphi^*\delta_G$
is an endomorphism of $G$, it follows that
$\delta_G^{-1}\varphi^*\delta_G$ is an automorphism.

The argument used at the end of the proof of Theorem
\ref{thm:cover_gives_envelope} now shows that $\varphi$ is an
automorphism as desired.
\end{proof}

\section{Existence of preenvelopes and the Govorov-Lazard Theorem}
\label{sec:preenvelopes}

This section proves Theorems A and B of the introduction; see Theorems
\ref{thm:limG} and \ref{thm:env}.

\begin{Setup}
In this section, the commutative noetherian ring $R$ is assumed to be
local with residue class field $k$.  We write $d = \depth R$.
\end{Setup}

In the following lemma, the case $d = 0$ is trivial, $d = 1$ is
closely inspired by a proof of Takahashi, and $d \geqslant 2$ is
classical.  Recall that $\Omega^d(k)$ denotes the $d$th syzygy in a
minimal free resolution of $k$ over $R$.

\begin{Lemma}
\label{lem:syzygies}
There exists an $M$ in $\mod\, R$ such that $\Omega^d(k)$ is
isomorphic to a direct summand of $M^*$.
\end{Lemma}

\begin{proof}
$d = 0$.  We can use $M = k$ since $\Omega^d(k) = \Omega^0(k) = k$ and
since $M^* = \Hom(k,R) \cong k^e$ with $e \neq 0$ because $d = 0$.

$d = 1$.  We will show that $M = \Omega^d(k)^*$ works here; in fact,
we will show that the biduality homomorphism for $\Omega^d(k)$ is an
isomorphism so $\Omega^d(k) \cong \Omega^d(k)^{**} = M^*$.

There is a short exact sequence
\begin{equation}
\label{equ:f}
  0
  \rightarrow \fm
  \stackrel{\mu}{\rightarrow} R
  \rightarrow k
  \rightarrow 0
\end{equation}
where $\fm$ is the maximal ideal of $R$ and $\mu$ is the inclusion, so
$\Omega^d(k) = \Omega^1(k) = \fm$.

If $R$ is regular then $k$ has projective dimension $1$ by the
Auslander-Buchsbaum formula, so \eqref{equ:f} shows that $\fm$ is
projective whence the biduality homomorphism $\delta_{\fm}$ is an
isomorphism as desired. 

Assume that $R$ is not regular.  For reasons of clarity, we start by
reproducing, in our notation, part of Takahashi's proof of
\cite[thm.\ 2.8]{Takahashi1}.  Applying $(-)^*$ and its derived
functors to the short exact sequence \eqref{equ:f} gives a long exact
sequence containing
\begin{equation}
\label{equ:e}
  0 \rightarrow R^* \stackrel{\mu^*}{\rightarrow} \fm^* \rightarrow
  k^e \rightarrow 0
\end{equation}
where we have written $k^e$ instead of $\Ext^1(k,R)$, and where $e
\neq 0$ since $d = 1$.  Applying $(-)^*$ again gives a left exact
sequence $0 \rightarrow (k^e)^* \rightarrow \fm^{**}
\stackrel{\mu^{**}}{\rightarrow} R^{**}$; here $(k^e)^* = 0$ because
$d = 1$, so $\mu^{**}$ is injective.

Consider the commutative square
\[
  \xymatrix{
    \fm \ar@{^(->}[r]^{\mu} \ar@{^(->}[d]_{\delta_{\fm}} & R \ar[d]_{\cong}^{\delta_R} \\
    \fm^{**} \ar@{^(->}[r]_{\mu^{**}} & R^{**}
           }
\]
where $\delta_{\fm}$ is injective because $\delta_R \mu$ is injective.
There are inclusions
\begin{equation}
\label{equ:h}
  \Image(\mu^{**}\delta_{\fm})
  \subseteq \Image(\mu^{**})
  \subseteq R^{**}.
\end{equation}
We have $R^{**}/\Image(\mu^{**}\delta_{\fm}) =
R^{**}/\Image(\delta_R\mu) \cong R/\Image(\mu) \cong k$ where the
first $\cong$ is because $\delta_R$ is an isomorphism.  This quotient
is simple so one of the inclusions \eqref{equ:h} must be an equality;
this means that either $\mu^{**}$ or $\delta_{\fm}$ is an isomorphism.
Suppose that $\mu^{**}$ is an isomorphism; we will prove a
contradiction whence $\delta_{\fm}$ is an isomorphism as desired.

To get the contradiction, we now depart from Takahashi's proof.  Since
$\mu^{**}$ is an isomorphism, so is $R^{***}
\stackrel{\mu^{***}}{\longrightarrow} \fm^{***}$, and so $\fm^{***}
\cong R$.  But $(\delta_{\fm})^*\delta_{\fm^*} = \id_{\fm^*}$ by the
proof of \cite[prop.\ (1.1.9)]{LWC}, so $\fm^*
\stackrel{\delta_{\fm^*}}{\longrightarrow} \fm^{***}$ is a split
monomorphism.  It follows that $\fm^*$ is a direct summand of $R$, so
$\fm^*$ is projective.  Hence the exact sequence \eqref{equ:e} gives a
projective re\-so\-lu\-ti\-on of $k^e$, and since $e \neq 0$ it
follows that $\gldim R \leqslant 1$ contradicting that $R$ is not
regular. 

$d \geqslant 2$.  Here we have $\Omega^d(k) = \Omega^2(\Omega^{d-2}(k))$,
so it is enough to show that a second syzygy of a finitely generated
module is a direct summand of some $M^*$.  In fact, such a second
syzygy $\Omega^2$ is isomorphic to an $M^*$.  Namely, $\Omega^2$ sits
in a short exact sequence $0 \rightarrow \Omega^2
\rightarrow P \stackrel{\pi}{\rightarrow} Q$ where $P$ and $Q$ are finitely
generated projective modules.  Consider the right-exact sequence $Q^*
\stackrel{\pi^*}{\rightarrow} P^* \rightarrow M \rightarrow 0$ and
apply $(-)^*$ to get a left-exact sequence $0 \rightarrow M^*
\rightarrow P^{**} \stackrel{\pi^{**}}{\longrightarrow} Q^{**}$.
Since $\pi^{**}$ is isomorphic to $\pi$, we get $\Omega^2 \cong M^*$.
\end{proof}

The following lemma is implicitly in \cite{CPST}, but it is handy to
make it explicit for reference.  Recall from \cite[defs.\ (2.1)]{CPST}
that if $\cB$ is a full subcategory of $\mod\, R$, then a
$\cB$-approximation of an $M$ in $\mod\, R$ is a short exact sequence
$0 \rightarrow K \rightarrow B \rightarrow M
\rightarrow 0$ where $B$ is in $\cB$ and $\Ext_R^{\geqslant 1}(\cB,K) =
0$.

\begin{Lemma}
\label{lem:special_precover_is_approximation}
Consider a special $\cG$-precover and complete it with its kernel.
The resulting short exact sequence $0 \rightarrow K \rightarrow G
\rightarrow M \rightarrow 0$ is a $\cG$-approximation of $M$.
\end{Lemma}

\begin{proof}
We know $\Ext^1(\cG,K) = 0$.  By \cite[cor.\ (4.3.5)(a)]{LWC} each $G$
in $\cG$ sits in a short exact sequence $0 \rightarrow G^{\prime}
\rightarrow P \rightarrow G \rightarrow 0$ where $P$ is a finitely
generated projective module and $G^{\prime}$ is in $\cG$, and it
follows by an easy induction that $\Ext^{\geqslant 1}(\cG,K) = 0$ as
desired.
\end{proof}

\begin{Remark}
\label{rmk:CPST}
Let us give a brief summary of a part of \cite{CPST}.

Recall from \cite[(1.1)]{CPST} that if $\cB$ is a full subcategory of
$\mod\, R$, then $\langle \cB \rangle$ denotes the closure under
direct summands and extensions.  The class of finitely generated
Gorenstein projective modules $\cG$ is a so-called reflexive
subcategory of $\mod\, R$ by \cite[def.\ (2.6)]{CPST}.  It follows
from \cite[prop.\ (2.10)]{CPST} that $\langle \widehat{R} \otimes_R
\cG \rangle$ is a reflexive subcategory of $\mod\, \widehat{R}$.

Now suppose that there is an $\langle \widehat{R} \otimes_R \cG
\rangle$-cover of $\Omega_{\widehat{R}}^d(k)$.  The cover is an
$\langle \widehat{R} \otimes_R \cG \rangle$-approximation by
\cite[(2.2)(b)]{CPST}.  But when such an approximation exists, the
proof of \cite[thm.\ (3.4)]{CPST} gives that either, $\widehat{R}$ is
Gorenstein, or $\langle \widehat{R} \otimes_R \cG \rangle$ consists of
free $\widehat{R}$-modules.
\end{Remark}

An important input to the proof of the next theorem are the methods
and results developed by Christensen, Piepmeyer, Striuli, and Takahashi
in \cite{CPST}.

\begin{Theorem}
\label{thm:env}
The following three conditions are equivalent.
\begin{enumerate}

  \item  Each module in $\mod\, R$ has a $\cG$-envelope.
         
\smallskip

  \item  Each module in $\mod\, R$ has a special $\cG$-preenvelope.
         
\smallskip

  \item  $R$ is Gorenstein or $\cG = \cF$.

\end{enumerate}
They imply the following condition.
 \begin{enumerate}
\setcounter{enumi}{3}

  \item  Each module in $\mod\, R$ has a $\cG$-preenvelope.
         
\end{enumerate}
Moreover, if $R$ is henselian then {\rm (iv)} implies {\rm (i)}, {\rm
(ii)}, and {\rm (iii)}.
\end{Theorem}

\begin{proof}
(i)$\Rightarrow$(ii).  Holds by Wakamatsu's Lemma, \cite[lem.\
2.1.2]{Xu}.

(ii)$\Rightarrow$(iii).  By Lemma \ref{lem:syzygies} the module
$\Omega_R^d(k)$ is a direct summand in a module of the form $M^*$
where $M$ is in $\mod\, R$.  If (ii) holds then $M$ has a special
$\cG$-preenvelope, and by Theorem \ref{thm:envelope_gives_cover}(ii)
it follows that $M^*$ has a special $\cG$-precover.  Completing with
the kernel gives a short exact sequence $0
\rightarrow K \rightarrow G \rightarrow M^*
\rightarrow 0$ which is a $\cG$-approximation of $M^*$ by Lemma
\ref{lem:special_precover_is_approximation}.

Tensoring the sequence with $\widehat{R}$ gives an $\langle
\widehat{R} \otimes_R \cG \rangle$-approximation of $\widehat{R}
\otimes_R M^*$ by \cite[prop.\ 2.4]{CPST}.  In particular, there is an
$\langle \widehat{R} \otimes_R \cG \rangle$-precover of $\widehat{R}
\otimes_R M^*$, and the same must hold for its direct summand
$\widehat{R} \otimes_R \Omega_R^d(k) \cong \Omega_{\widehat{R}}^d(k)$.
Hence there is an $\langle \widehat{R} \otimes_R \cG \rangle$-cover of
$\Omega_{\widehat{R}}^d(k)$ by \cite[cor.\ 2.5]{Takahashi3}.

But now the results of \cite{CPST} imply that either, $\widehat{R}$ is
Gorenstein, or $\langle \widehat{R} \otimes_R \cG \rangle$ consists of
free $\widehat{R}$-modules; see Remark \ref{rmk:CPST}.  In the former
case, $R$ is Gorenstein by \cite[thm.\ 18.3]{Matsumura}.  In the
latter case, in particular, $\widehat{R} \otimes_R G$ is a free
$\widehat{R}$-module whenever $G$ is in $\cG$.  But then $G$ is a free
$R$-module whence $\cG = \cF$; cf.\ \cite[cor.\ p.\ 53, exer.\ 7.1,
and (3), p.\ 63]{Matsumura}.

(iii)$\Rightarrow$(i).  First, suppose that $R$ is Gorenstein.  Then
each finitely ge\-ne\-ra\-ted $R$-module has a $\cG$-cover by
unpublished work of Auslander; see \cite[thm.\ 5.5]{EJX}.  Existence
of $\cG$-envelopes now follows from Theorem
\ref{thm:cover_gives_envelope}(iii).

Secondly, suppose $\cG = \cF$.  Then each finitely generated
$R$-module has an $\cF$-envelope by \cite[Prop.\ 2.3(3)]{Takahashi2},
which does not need that paper's assumption that the ring is
henselian.

(i)$\Rightarrow$(iv).  Trivial.

Now assume that $R$ is henselian.

(iv)$\Rightarrow$(i).  Suppose that (iv) holds.  Then Theorem
\ref{thm:envelope_gives_cover}(i) implies that each $R$-module of the
form $M^*$ with $M$ in $\mod\, R$ has a $\cG$-precover.  Since $R$ is
henselian, each $M^*$ has a $\cG$-cover by \cite[cor.\
2.5]{Takahashi3}, and so each $M$ has a $\cG$-envelope by Theorem
\ref{thm:cover_gives_envelope}(iii).
\end{proof}

\begin{Remark}
\label{rmk:main}
As a consequence, the following conditions are e\-qui\-va\-lent. 
\begin{enumerate}

  \item  Each module in $\mod\, R$ has a $\cG$-cover.

\smallskip

  \item  Each module in $\mod\, R$ has a special $\cG$-precover.

\smallskip

  \item  Each module in $\mod\, R$ has a $\cG$-envelope.

\smallskip

  \item  Each module in $\mod\, R$ has a special $\cG$-preenvelope.

\smallskip

  \item  $R$ is Gorenstein or $\cG = \cF$.

\end{enumerate}
Namely, (i)$\Rightarrow$(ii) is by Wakamatsu's Lemma, \cite[lem.\
2.1.1]{Xu}.  (ii)$\Rightarrow$(iv) follows from Theorem
\ref{thm:cover_gives_envelope}(ii).  Conditions (iii), (iv), and (v)
are e\-qui\-va\-lent by Theorem \ref{thm:env}.  And (v)$\Rightarrow$(i)
follows from unpublished work by Auslander; see \cite[thm.\
5.5]{EJX}.

Note that the equivalence of (i), (ii), and (v) was first established
in \cite{CPST}, and that our proof depends on that paper.

Now assume that $R$ is henselian.  Combining with a result of
Crawley-Boevey shows that the following conditions are also
equivalent, where $\varinjlim \cG$ denotes the closure of $\cG$ under
direct limits.
\begin{enumerate}

  \item  Each module in $\mod\, R$ has a $\cG$-precover.

\smallskip

  \item  Each module in $\mod\, R$ has a $\cG$-preenvelope.

\smallskip

  \item  $R$ is Gorenstein or $\cG = \cF$.

\smallskip

  \item  $\varinjlim \cG$ is closed under set indexed direct products.

\end{enumerate}
Namely, (i)$\Rightarrow$(iii) holds by \cite[(2.8) and thm.\
(3.4)]{CPST}.  (iii)$\Rightarrow$(i) follows from unpublished work by
Auslander as above; see \cite[thm.\ 5.5]{EJX}.
(ii)$\Leftrightarrow$(iii) is by Theorem \ref{thm:env}.  And
(ii)$\Leftrightarrow$(iv) holds by
\cite[(4.2)]{WCB}.
\end{Remark}

\begin{Theorem}
\label{thm:limG}
If $R$ has a dualizing complex, is henselian, not Gorenstein, and has
$\cG \neq \cF$, then $\varinjlim \cG$ is strictly contained in the
class of Gorenstein flat modules.
\end{Theorem}

\begin{proof}
Each module in $\cG$ is Gorenstein flat, cf.\ \cite[Thm.\
(5.1.11)]{LWC}, and the class of Gorenstein flat modules is
closed under direct limits by \cite{ELR}, so $\varinjlim \cG$ is
contained in the class of Gorenstein flat modules.

The class of Gorenstein flat modules is closed under set indexed
pro\-ducts by \cite[thm.\ 5.7]{CFH}.  On the other hand, by the last
four conditions of Remark \ref{rmk:main}, the assumptions on $R$ imply
that $\varinjlim \cG$ is not closed under set indexed products.
\end{proof}

\begin{Example}
\label{exa:dual_numbers}
It is easy to find rings of the type required by Theorem
\ref{thm:limG}.  For instance, let us show that the $1$-dimensional
ring 
\[
  T
  = \BQ\mbox{\textlbrackdbl} X,Y,Z,W \mbox{\textrbrackdbl} / (X^2,Y^2,Z^2,XY)
\]
satisfies the conditions of the theorem.

First note that since $T$ is complete, it has a dualizing complex and
is henselian.

Next consider $S = \BQ\mbox{\textlbrackdbl} X,Y
\mbox{\textrbrackdbl}/(X^2,Y^2,XY)$ which is not Gorenstein.  The ring
$T$ is $S\mbox{\textlbrackdbl} Z,W
\mbox{\textrbrackdbl}/(Z^2)$; that is, $T$ is the ring of dual numbers
over $S\mbox{\textlbrackdbl} W \mbox{\textrbrackdbl}$.  Since $S$ is
not Gorenstein, neither is $S\mbox{\textlbrackdbl} W
\mbox{\textrbrackdbl}$ or $T$.

Finally, let $z$ be the image of $Z$ in $T$.  Then the complete
projective resolution
\[
  \cdots \rightarrow T \stackrel{z \cdot}{\rightarrow} T \stackrel{z
  \cdot}{\rightarrow} T \rightarrow \cdots
\]
shows that the non-projective module $T/(z)$ is Gorenstein
projective, so $\cG \neq \cF$.
\end{Example}

\begin{Remark}
\label{rmk:BK}
Assume that $R$ is artinian.  Then it has a dualizing complex and is
henselian (in fact, $R$ is complete).  Moreover, it is easy to prove
that each Gorenstein flat module is Gorenstein projective.

If $R$ is not Gorenstein and has $\cG \neq \cF$, then Theorem
\ref{thm:limG} shows that $\varinjlim \cG$ is strictly contained in
the class of Gorenstein projective modules.  Hence \cite[4.2]{BK}
shows that $R$ is not a so-called virtually Gorenstein ring.
\end{Remark}

\section{A special Gorenstein flat module}
\label{sec:special}

This short section shows a method for constructing a Gorenstein flat
module outside $\varinjlim \cG$.

\begin{Construction}
\label{con:product}
Let $\{G_i\}_{i \in I}$ be a set of representatives of the isomorphism
classes of indecomposable modules in $\cG$.  Let $M$ be in $\mod\, R$.
For each $i$ in $I$, view $H(i) = \Hom_R(M,G_i)$ as a set and consider
the direct product $G_i^{H(i)}$ indexed by that set.  Define
\[
  \Lambda(M) = {\textstyle \prod_{i \in I}} G_i^{H(i)}.
\]
\end{Construction}

\begin{Proposition}
\label{pro:product}
Assume that $R$ has a dualizing complex.  Let $M$ be in $\mod\, R$ and
suppose that $M$ does not have a $\cG$-preenvelope.  Then $\Lambda(M)$
is a Gorenstein flat module outside $\varinjlim \cG$.
\end{Proposition}

\begin{proof}
As in the proof of Theorem \ref{thm:limG}, the modules in $\cG$
are Gorenstein flat and the class of Gorenstein flat modules is closed
under set indexed products, so $\Lambda(M)$ is Gorenstein flat.

For each $i$ in $I$, consider the homomorphism
\[
    M \stackrel{\mu_i}{\rightarrow} G_i^{H(i)}, \;\;\; m \mapsto (h(m))_{h \in H(i)}. 
\]
Let $\Lambda(M) \stackrel{\pi_i}{\rightarrow} G_i^{H(i)}$ be the
$i$'th projection, and let $M \stackrel{\mu}{\rightarrow} \Lambda(M)$
be the unique homomorphism which satisfies $\pi_i\mu = \mu_i$ for each
$i$ in $I$.  Then each homomorphism $M \stackrel{\eta}{\rightarrow} G$
with $G$ in $\cG$ factors through $\mu$,
\[
  \xymatrix{ 
    M \ar[d]_-{\eta} \ar[r]^-{\mu} & \Lambda(M). \ar@{-->}^{\lambda}[dl] \\ 
    G & {}
           }
\]
Namely, we may assume $G = G_i$ for some $i$, since each $G$ in $\cG$
is isomorphic to a finite direct sum of modules from the set
$\{G_i\}_{i \in I}$.  But then $\eta$ is an element of $H(i)$,
and we can let $\lambda$ equal the composition of the projections
$\Lambda(M) \stackrel{\pi_i}{\rightarrow} G_i^{H(i)} \rightarrow G_i$
where the second one is onto the $\eta$th copy of $G_i$.

Now, $M$ is finitely presented, so if $\Lambda(M)$ were in
$\varinjlim\cG$ then \cite[prop.\ 2.1]{Lenzing} would give that $\mu$
could be factored as $M \stackrel{\widetilde{\mu}}{\rightarrow}
\widetilde{G} \rightarrow \Lambda(M)$ with $\widetilde{G}$ in $\cG$.
Since each homomorphism $M \stackrel{\eta}{\rightarrow} G$ factors
through $\mu$ by the above, it would also factor through
$\widetilde{\mu}$ which would hence be a $\cG$-preenvelope of $M$.
Since there is no such $\cG$-preenvelope, $\Lambda(M)$ is outside
$\varinjlim \cG$.
\end{proof}

\medskip
\noindent
{\em Acknowledgement. }  We thank Luchezar Avramov, Apostolos
Beligiannis, and Lars Win\-ther Christensen for comments to previous
versions of the paper.  In particular, LA pointed out reference
\cite{Govorov}, and AB informed us that \cite{BK} had already found
Artin algebras without a Gorenstein analogue of the Govorov-Lazard
Theorem.

\end{document}